\documentclass[11pt]{amsart}
\usepackage{amsmath,amsthm, amscd, amssymb, amsfonts}
\usepackage{hhline}

\usepackage[all]{xy}

\usepackage{pictexwd,dcpic} 
\usepackage{pslatex}
\usepackage{mathdots}
\usepackage{fancyvrb}
\usepackage{t1enc}
\usepackage{latexsym}
\usepackage[latin1]{inputenc} 
\usepackage{subfigure}    
\usepackage{verbatim}

\newcommand{\R}{{\mathcal R}}

\newcommand{\C}{{\mathbb C}}

\newcommand{\T}{{\tt t}}

\newcommand{\End}{\operatorname{End}}
\newcommand{\Aut}{\operatorname{Aut}}

\newcommand\ad{\operatorname{ad}}
\newcommand\Hom{\operatorname{Hom}}

\numberwithin{equation}{section}\theoremstyle{plain}

\newtheorem{theorem}{Theorem}[section]
\newtheorem{lemma}[theorem]{Lemma}
\newtheorem{corollary}[theorem]{Corollary}

\newtheorem{proposition}[theorem]{Proposition}

\theoremstyle{definition}

\theoremstyle{remark}

\def\pf{\begin{proof}}
\def\epf{\end{proof}}

\theoremstyle{remark}

\def\h4n{\hspace{-0.4cm}}

\def\R{\mathbb{R}}

\def\[{|\hspace{-0.2ex} [}
\def\]{]\hspace{-0.2ex} |}

\newcommand\g{\mathfrak g}

\renewcommand\h{\mathfrak h}

\renewcommand\-[1]{\,-_{\!\scriptscriptstyle #1}\,}

\renewcommand\End{\operatorname{End}}

\newcommand\Lie{\operatorname{Lie}}
\newcommand\Cl{\operatorname{Cl}}

\begin{document}

\renewcommand{\baselinestretch}{1.2}
\thispagestyle{empty}

\title[Automorphisms of non-singular nilpotent Lie algebras]{Automorphisms of non-singular nilpotent Lie algebras}
\author[Kaplan and Tiraboschi]{Aroldo Kaplan}
\address{\noindent
Facultad de Matem\'atica, Astronom\'\i a y F\'\i sica, Universidad
Nacional de C\'ordoba.  CIEM -- CONICET. (5000) Ciudad
Universitaria, C\'ordoba, Argentina}
 \email{(kaplan, tirabo)@famaf.unc.edu.ar}
\author[]{Alejandro Tiraboschi}
\thanks{This work was partially supported by CONICET, ANPCyT and Secyt (UNC) (Argentina)}
\date{\today}
\begin{abstract} For a real, non-singular, 2-step nilpotent Lie algebra $\mathfrak{n}$, the group $\Aut(\mathfrak{n})/\Aut_0(\mathfrak{n})$, where $\Aut_0(\mathfrak{n})$ is the group of automorphisms which act trivially on the center, is the direct product of a compact group with the 1-dimensional group of dilations. Maximality of some automorphisms groups of $\mathfrak{n}$ follows and is related to how close is $\mathfrak{n}$ to being of Heisenberg type. For example, at least when the dimension of the center is two, $\dim \Aut(\mathfrak{n})$ is maximal if and only if $\mathfrak{n}$ is type $H$. The connection with fat distributions is discussed.
\end{abstract}

\maketitle

\begin{section}{Introduction}\label{section1} A 2-step nilpotent real Lie algebra  $\mathfrak{n}$ with center $\mathfrak{z}$  is called {\it non-singular} \cite{E} if  $\ad x:  \mathfrak{n}\rightarrow  \mathfrak{z}$ is onto for any $x\notin\mathfrak{z}$. Equivalently, it is a vector-valued antisymmetric form
$$[\ ,\ ]: \mathfrak{v}\times \mathfrak{v}\rightarrow \mathfrak{z},$$
 $\mathfrak{v}=\mathfrak{n}/\mathfrak{z}$, such that the 2-forms $\lambda([u,v])$  on $\mathfrak{v}$ are non-degenerate for all $\lambda\in\mathfrak{z}^*$, $\lambda\not=0$.
We shall call such Lie algebras {\it fat algebras} for short, since they are the nilpotentizations, or symbols, of fat vector distributions.
While for $m=1$ there is  only one fat algebra up to isomorphisms, for $m\geq 2$ there is an uncountable number of isomorphism classes  and for $m\geq 3$ they form a wild set.

The group of automorphisms $\Aut(\mathfrak{n})$ is the semidirect product of the group $G(\mathfrak{n})$ of graded automorphisms of $\mathfrak{n}=\mathfrak{v}\oplus\mathfrak{z}$ with the abelian group $\Hom(\mathfrak{v},\mathfrak{z})$ times the group of dilations $(t,t^2)$.  Hence, we concentrate on $G(\mathfrak{n})$.
We prove that there is an exact sequence
$$
1\rightarrow G_0  \rightarrow  G \rightarrow O(m)
$$
where $G_0$ is the subgroup of $G $ of elements that act trivially on the center and $m$ is the dimension of this center.
In other words, there are positive metrics (inner products) on $\mathfrak{z}$ which are invariant under all of $\Aut(\mathfrak{n})$. If a metric $g$ is also given on $\mathfrak{v}$, as in the case of the nilpotentization of a  subriemannian structure, we also consider the subgroups $K_0$, $K $, of graded automorphisms that leave $g$ invariant, which define a compatible exact sequence
 $$
1\rightarrow K_0  \rightarrow  K \rightarrow O(m).
$$
Next, we compute the terms in this sequence and the images $G/G_0$ and $K/K_0$, proving that the exactness of
$$
1\rightarrow \Lie(K_0)  \rightarrow  \Lie(K) \rightarrow \mathfrak{so}(m)\rightarrow 1
$$
is equivalent to $\mathfrak{n}$ being of Heisenberg type, while the exactness of
$$
1\rightarrow \Lie(G_0)  \rightarrow  \Lie( G) \rightarrow \mathfrak{so}(m)\rightarrow 1
$$
is strictly more general.  As to $G_0(\mathfrak{n})$, we describe it in detail for the case $m=2$, leading a proof that, at least in that case, $\dim \Aut(\mathfrak{n})$ is maximal if and only if $\mathfrak{n}$ is of Heisenberg.

In the last section we explain the connection with the Equivalence Problem for fat subriemannian distributions.

Algebras of Heisenberg type, or {\it $H$-type}, arise as follows \cite{K}. If $\mathfrak{v}$ is a real unitary module over the Clifford algebra $\Cl(\mathfrak{z})$ associated to a quadratic form on $\mathfrak{z}$, the identity
$$<z,[u,v]>_\mathfrak{z} = <z\cdot u,v>_\mathfrak{v}$$
with $z\in  \mathfrak{z}\subset \Cl(\mathfrak{z})$, $u,v\in  \mathfrak{v}$,
defines a fat $[\ ,\ ]: \mathfrak{v}\times \mathfrak{v}\rightarrow \mathfrak{z}$. Alternatively, they are characterized by possessing a positive-definite metric such that the operator $z\cdot$ defined by the above equation satisfies $ z\cdot(z\cdot v)= - |z|^2 v$.

It follows from Adam's theorem on  frames on spheres \cite{H} that for any fat algebra there is an $H$-type algebra with the same $\dim\mathfrak{z}$ and $\dim\mathfrak{v}$. That these were, in some sense, the most symmetric, was expected from the properties of their sublaplacians \cite{BTV} \cite{CGN} \cite{GV} \cite{K}, but we found no explicit statements in this regard.
Finally, although the arguments below can be made more intrinsic, matrices are emphasized because they can be fed easily into MAGMA for the application of the methods of \cite{DG}.

\end{section}

\begin{section}{Automorphisms of fat algebras}\label{section2} Let $\mathfrak{n}$ be a 2-step Lie algebra with center $\mathfrak{z}$ and let $\mathfrak{v = n/z}$, so that
\begin{equation}\label{I}
\mathfrak{n}\cong \mathfrak{v}\oplus \mathfrak{z}
\end{equation}
and the Lie algebra structure is encoded into the map
$$
[\ ,\ ]: \Lambda^2\mathfrak{v}\rightarrow \mathfrak{z}.
$$
Let $n=\dim \mathfrak{v}$ and $m=\dim \mathfrak{z}$. Relative to a basis compatible with (\ref{I}), the bracket becomes an $\R^m$-valued antisymmetric form on $\R^n$ and an automorphism is a matrix of the form
$$\begin{pmatrix} a & 0 \\  c & b \end{pmatrix},\qquad a\in GL(n),\ b\in GL(m),\ c\in M_{n\times m}(\mathbb R)$$
such that
$$b([u,v])=[au,av].$$
$\Aut(\mathfrak{n})$ always contains the normal subgroup $\mathfrak{D(n)}$ of dilations and translations
$$\begin{pmatrix} tI_n & 0 \\  c & t^2I_m \end{pmatrix},\qquad t\in \R^*, \ c\in M_{n\times m}(\mathbb R).$$
Let
$$G=G(\mathfrak{n})=\{\begin{pmatrix} a & 0\\  0 & b \end{pmatrix},\  a\in SL(n),\ b\in GL(m),\ b([u,v])=[au,av]\}.$$
Then $\Aut(\mathfrak{n})$ is the semidirect product of $G(\mathfrak{n })$ with $\mathfrak{D(n)}$.
Let
$$G_0= G_0(\mathfrak{n}) = \{\begin{pmatrix} a & 0\\  0 & I_m \end{pmatrix},\ a\in SL(n),\  [au,av]=[u,v]  \},  $$
 the subgroup of automorphisms that act trivially on the center. These are Lie groups, $G_0$ is normal in $G$, and the quotient group
$$G/G_0 $$
can be identified with  the group of  $b\in GL(\mathfrak{z})$ such that
 $ b([u,v])=[au,av] $ for some $a\in SL(\mathfrak{v})$. Obviously,
\begin{equation}\label{igualdad}
\dim \Aut(\mathfrak{n}) = nm+1 + \dim(G/G_0) + \dim(G_0).
\end{equation}

\

 \begin{theorem}\label{th2.1}  Let  $\mathfrak{n}$ be a fat algebra with center $\mathfrak{z}$. Then there is a positive definite metric on $\mathfrak{z}$ invariant under $G(\mathfrak{n})$.
\end{theorem}
\begin{proof}
Fix arbitrary positive inner products on $\mathfrak{v}$ and $\mathfrak{z}$. For $z\in \mathfrak{z}$, $u,v\in \mathfrak{v}$
$$(T_zu,v)_\mathfrak{v} = (z,[u,v])_\mathfrak{z}$$
defines a linear map $z\mapsto T_z$ from $\mathfrak{z}$ to  $\End(\mathfrak{v})$. Clearly,
$$\mathfrak{n}\ fat\ \Leftrightarrow\   T_z\in GL(\mathfrak{v}) \ \forall z\not=0.$$
Hence the hypothesis insures that the Pfaffian
$$P(z)=\det(T_z)$$
is  non-zero   on $\mathfrak{z}\setminus \{0\}$. This is a homogeneous polynomial of degree $n$, so it satisfies
\begin{equation}\label{I'}k \|z\|^n\leq |P(z)| \leq K\|z\|^n
\end{equation}
where $k,K$ are the minimum and maximum values of $|P|$ on the unit sphere, which are positive.

Let now $g_{a,b} :=\begin{pmatrix} a & 0\\  0 & b \end{pmatrix}\in \Aut(\mathfrak{n})$. Then
$$T_{ b^\T z}=a^\T T_za$$
because
 $(T_{ b^\T z}u,v)_\mathfrak{v} = ( b^\T z,[u,v])_\mathfrak{z} = ( z,b([u,v]))_\mathfrak{z} = ( z,[au,av])_\mathfrak{z} = (T_zau,av)_\mathfrak{v} = ( a^\T T_zau,v)_\mathfrak{v}.$
Consequently
$$
P( b^\T z) = (\det a)^2 P(z).
$$
In particular, if $g\in G$ then $P( b^\T z) =  P(z)$.
This implies
$$
k \|b^\T z\|^n\leq |P(b^\T z)|= |P(z)| \leq K\|z\|^n
$$
for all $z$, therefore
 $\|b\| \leq  \sqrt[n]{K/k}$.
The group  of $b\in GL(\mathfrak{z})$ such that $g_{a,b}\in \Aut(\mathfrak{n})$ for some $a\in SL(\mathfrak{v})$, is therefore bounded in $\End( \mathbb{R}^m)$. Its closure is a compact Lie subgroup of $GL(\mathfrak{z})$, necessarily contained in  $O(\mathfrak{z})$ for some positive definite metric.

\end{proof}

From now on  $\mathfrak{z}$ will be assumed endowed with such invariant metric.
If a metric $g$ on $\mathfrak{v}$ is also fixed, as in the case of the nilpotentization of a subriemannian structure,   define
the groups
$$ K= K(\mathfrak{n},g) = \{\begin{pmatrix} a & 0\\  0 & b \end{pmatrix},\  a\in SO(\mathfrak{v}),\ b\in O(\mathfrak{\ z}),\ [au,av]=b[u,v]  \}$$
$$ K_0=K_0(\mathfrak{n},g) = \{\begin{pmatrix} a & 0\\  0 & I \end{pmatrix},\  a\in SO(\mathfrak{v}),\ [au,av]=[u,v]  \}.$$

Let $\mathfrak{g,g_0,k,k_0}$ be the Lie algebras of $G,G_0,K,K_0$ respectively. Then  there is the commutative diagram with exact rows
\begin{equation*}\label{ancla}
\begindc{0}[3]
 \obj(10,20)[A11]{$0$}
 \obj(20,20)[A12]{$\mathfrak{g}_0$}
 \obj(30,20)[A13]{$\mathfrak{g}$}
 \obj(42,20)[A14]{$\mathfrak{so}(m)$}
 \obj(10,10)[A21]{$0$}
 \obj(20,10)[A22]{$\mathfrak{k}_0$}
 \obj(30,10)[A23]{$\mathfrak{k}$}
 \obj(42,10)[A24]{$ \mathfrak{so}(m)$}
 \mor{A11}{A12}{}
 \mor{A21}{A22}{}
 \mor{A22}{A12}{}[\atleft,\solidarrow]
 \mor{A12}{A13}{}
 \mor{A13}{A14}{}
 \mor{A22}{A23}{}
 \mor{A23}{A24}{}
 \mor{A23}{A13}{}
 \mor{A24}{A14}{}
\enddc
\end{equation*}
where the vertical arrows are the inclusions.
Below we prove that the bottom sequence extends to
$$0\rightarrow \mathfrak{k}_0\ \rightarrow \mathfrak{k}\rightarrow \mathfrak{so}(m)\rightarrow 0$$
if and only if $\mathfrak{n}$ is type $H$. This is not the case for the top one: the condition that
$$0\rightarrow \mathfrak{g}_0\ \rightarrow \mathfrak{g}\rightarrow \mathfrak{so}(m)\rightarrow 0$$
is exact defines  a class of fat algebras strictly larger than type $H$.  We describe it in the next section for  $m=2$.

  \begin{proposition}\label{theorem2.2} Let  $\mathfrak{n}=\mathfrak{v}+ \mathfrak{z}$ be an algebra of type $H$. There is a metric on $\mathfrak{z}$ such that
 $\mathfrak{g/g_0 \cong so}(m)$.
\end{proposition}
\begin{proof}
There is an inner product in $\mathfrak{v}$ such that the $J_i=T_i$'s satisfy the Canonical Anticommutation Relations
 $$J_w J_z+J_zJ_w = -2<z,w>I.$$
 For $\|z\|=1$ let $r_z\in O(\mathfrak{z})$ be the reflection through the hyperplane orthogonal to $z$ and $J_z\in SL(\mathfrak{v})$ be as above. Then
 $$g_{(J_z,-r_z)}=\begin{pmatrix} J_z & 0\\  0 & -r_z \end{pmatrix}\in \Aut(\mathfrak{n}).$$
 Indeed,
\begin{align*}
(w,[J_zu,J_zv])&= (J_w J_zu,J_zv ) = (-J_zJ_w u -2(z,w)u,J_zv)\\
&=  -( J_zJ_w u ,J_zv)-2(z,w)(u,J_zv) =   ( J_w u ,J_zJ_zv)+2(z,w)(J_zu,v)\\
&=  - ( J_w u , v)+2(z,w)(J_zu,v)=  ( J_{-w +2(z,w)z}u,v)\\
&=  (  -w +2(z,w)z,[u,v])= (-r_z(w),[u,v])\\
&= (w ,-r_z([u,v])),
\end{align*}
so that
$$-r_z([u,v])= [J_zu,J_zv].$$
The Lie group  generated by the $-r_z$ has finite index in
  $O(\mathfrak{z})$.
  \end{proof}

\begin{corollary}\label{corollary2.3} Let  $\mathfrak{n}$ be a fat algebra with center of dimension $m$. Then
$$\dim(K/K_0) \leq \dim (G/G_0) \leq m(m-1)/2 $$
with equality achieved for any type $H$ algebra of the same dimension with center of the same dimension.
\end{corollary}

\

Since $\Aut(\mathfrak{n})/Aut_0(\mathfrak{n}) = (G/G_0) \times$(dilations), one obtains

\

\begin{corollary} Let  $\mathfrak{n}$ be a fat algebra with center of dimension $m$. Then
$$ \dim(\Aut(\mathfrak{n})/Aut_0(\mathfrak{n})) \leq 1+ m(m-1)/2,$$
with equality achieved for any type $H$ algebra of the same dimension and with center of the same dimension.
\end{corollary}

\

A converse for Corollary \ref{corollary2.3} is

\

\begin{theorem}   If $\mathfrak{n}$ is fat with center of dimension $m$ and
$$\dim( K/K_0 )= m(m-1)/2$$
for some metric on $\mathfrak{v}$, then $\mathfrak{n}$
is of type $H$.
\end{theorem}
\begin{proof}
The hypothesis implies  that $ \mathfrak{k/k_0}=  \mathfrak{g/g_0}\cong \mathfrak{so}(m)$, so that
 $K/K_0$ acts transitively among the $|z|=1$. For $\begin{pmatrix} a & 0\\  0 & b \end{pmatrix}$ in this group,
$-T_{bz} = aT_{z}a^{-1}$, hence $T_{bz}^2 = aT_{z}^2a^{-1}$. Since $T_z$ is invertible, we can choose the metric such that $T_{z_0}^2=-I$ for any given $z_0$. Therefore $T_z^2=-I$ for all $|z|=1$, which implies the assertion.
\end{proof}

\

\

Maximal dimension means there are isomorphisms
$$\Lie(K/K_0) = \Lie(G/G_0) \cong \mathfrak{so} (m).$$
Therefore the simply connected covers are isomorphic:
 $Spin(m) \cong \widetilde{(G/G_0)_e}.$
The induced homomorphism
$$Spin(m) \rightarrow (G/G_0)_e$$
may or may not extend  to a homomorphism
$$Pin(m) \rightarrow   G/G_0 .$$
If it does extend, it may or may not be injective, in which case it is an isomorphism.
 Therefore, among the  algebras for which $\dim(G/G_0)$ is maximal, those for which
 $Pin(m) \cong G/G_0$
can be regarded as the most symmetric.

\begin{theorem}\label{thcliff} Suppose $\mathfrak{n}$ is a 2-step graded algebra such that $\Aut(\mathfrak{n})$ contains a copy of $Pin(m)$ inducing the standard action on $\mathfrak{z}$. Then
$\mathfrak{n}$ is type $H$.
\end{theorem}
\begin{proof}
The assumption implies that there is a linear map $\mathfrak{z}\rightarrow \End(\mathfrak{v})$, denoted by $z\mapsto J_z$ such that
 $J_z^2 = - |z|^2I$ for all $z$ and
 $$[J_zu,J_zv]= r_z([u,v])$$
 for $u,v\in \mathfrak{v}$, $z\in \mathfrak{z}$, $|z|=1$, where $r_z$ is the reflection in $\mathfrak{z}$ with respect of the line spanned by $z$. $Pin(m)$ is the group generated by the $J_z$'s with $\|z\|=1$, which acts linearly on $\mathfrak{v}$ and is compact. Fix a metric on $\mathfrak{v}$ invariant under it.

We get, as in the proof of Theorem \ref{th2.1},  that if $\begin{pmatrix} a & 0\\  0 & b \end{pmatrix}\in \Aut(\mathfrak{n})$, then
$$T_{ b^\T z}=a^\T T_za.$$
In particular:
$$T_{  r_x(z)}=J_x T_z J_x.$$
If $x=z$, we get $T_z = -J_zT_zJ_z$, thus $T_zJ_z = -J_z^{-1}T_z = J_zT_z$. If $x\perp z$, we get  $T_z = J_xT_zJ_x$, thus  $T_zJ_x = J_x^{-1}T_z = -J_xT_z$. It follows that $T_z^2$ commutes with $J_z$ and with $J_w$, $w\perp z$.

Now, let $z\in\mathfrak{z}$ and $w\perp z$. Let $R_w(t)$ the $2t$-rotation from $z$ towards $w$. Then $R_w(t) =  r_zr_{w(t)}$, with $w(t) =  \cos(t)z + \sin(t)w$. It follows that
$$
\begin{pmatrix} J_zJ_{w(t)} & 0\\  0 & R_w(t) \end{pmatrix}
$$
is an orthogonal automorphism and, therefore, satisfies
$$
T_{R_w(t)z} = (J_zJ_{w(t)})^\T T_z (J_zJ_{w(t)}).
$$
Since $(J_zJ_{w(t)})^\T = (J_zJ_{w(t)})^{-1}$,
$$
T_{R_w(t)z}^2 = (J_zJ_{w(t)})^\T T_{z}^2 (J_zJ_{w(t)}) =    J_{w(t)}J_zT_{z}^2J_zJ_{w(t)}.
$$
Since $T_z^2$ commutes with $J_z$ and $J_w$,
\begin{equation}\label{t2}
T_{R_w(t)z}^2 =  T_{z}^2J_{w(t)}J_zJ_zJ_{w(t)} =  -T_{z}^2J_{w(t)}J_{w(t)}.
\end{equation}
But $J_{w(t)}^2 = -I$, so that
(\ref{t2}) implies that
$$
T_{R_w(t)z}^2 = T_{z}^2.
$$
For all $z' \in  \mathfrak{z}$ we can choose $w\in \mathfrak{z},t \in \mathbb R$ such that $R_w(t)z = z'$, so we get
$$
T_{z'}^2 = T_{z}^2, \quad \text{ for all } z' \in \mathfrak{z}, |z'| =1.
$$
The anti-symmetry of the bracket implies that $T_z$ is skew-symmetric. Rescaling the scalar product on $\mathfrak{v}$ we obtain
that $T_{z}^2 = -I$, so $T_{z'}^2 = -I$  for all  $z' \in \mathfrak{z}$, $|z'| =1$. Therefore $\mathfrak{n}$ is type $H$.

\end{proof}

\end{section}

\begin{section}{The case of center of dimension 2}\label{section3}

In this section we compute the groups $G,G_0, G/G_0$ in the case $m=2$. The various types are parametrized by pairs
$$(\mathbf{c},\mathbf{r})
\in  (\mathbb{U}^\ell /SL(2,\mathbb{R}))
 \times  \mathbb{Z}_+^\ell $$
  where  $\mathbb{U}$ is the upper-half plane and $2\ell = 2\sum r_j
 = \dim \mathfrak{n}-2.
$
As a corollary we conclude that $\Aut(\mathfrak{n})$  is maximal if and only if $\mathfrak{n}$ is  type $H$. These are
complex Heisenberg algebras of various dimensions regarded as real Lie algebras.

First we recall the normal form for fat algebras with $m=2$ deduced from \cite{LT}. Given  $c= a + bi\in \mathbb{C} $, let
$$
Z(c)= \begin{pmatrix} a &b \\ -b &a \end{pmatrix}.
$$
If $r\in \mathbb{Z}_+$, set
$$
A(c,r) = \begin{pmatrix} Z(c) & & & \\ I_2  & Z(c) & & \\& & \ddots &  \\ & & I_2  & Z(c)\end{pmatrix}
$$
a $2r\times 2r$-matrix. If $\mathbf{c}=(c_1,...,c_\ell)\in \mathbb{C}^\ell$ and $\mathbf{r}=(r_1,...,r_\ell)\in \mathbb{N}_+^\ell$, set
$$
A(\mathbf{c},\mathbf{r})  = \begin{pmatrix} A(c_1,r_1)& & & \\   & A(c_2,r_2) & & \\& & \ddots &  \\ & &   & A(c_\ell,r_\ell) \end{pmatrix}
$$
which is a $2s\times 2s$ matrix, $s=r_1+...+r_\ell$.

Let now $\phi, \psi_{(\mathbf{c},\mathbf{r})}$ be the 2-forms on $\mathbb{R}^{4s}$ whose matrices in the standard basis are
\begin{equation}\label{psi} [\phi]=\begin{pmatrix} 0_{  }&- I_{ 2s } \\I_{2s  } & 0_{  } \end{pmatrix}\qquad
[\psi_{(\mathbf{c},\mathbf{r})}]= \begin{pmatrix} 0_{ }& A(\mathbf{c},\mathbf{r})  \\-A^\T(\mathbf{c},\mathbf{r}) & 0_{ } \end{pmatrix}.
\end{equation}
Then
$$
[u,v]_{(\mathbf{c},\mathbf{r})} = (\phi (u,v),\psi_{(\mathbf{c},\mathbf{r})} (u,v))=<u,[\phi]v>e_1 + <u,[\psi_{(\mathbf{c},\mathbf{r})}]v>e_2
$$
is an $\mathbb{R}^{2}$-valued antisymmetric 2-form on $\mathbb{R}^{4s}$. Let
$$
\mathfrak{n}_{(\mathbf{c},\mathbf{r})} = \mathbb{R}^{4s}\oplus \mathbb{R}^2
$$
be the corresponding Lie algebra.

Define $M_{(\mathbf{c},\mathbf{r})}\in\End(\mathfrak{v})$ by
$$
\phi(M_{(\mathbf{c},\mathbf{r})}u,v) = \psi_{(\mathbf{c},\mathbf{r})}(u,v),
$$
whose matrix is
$$
[M_{(\mathbf{c},\mathbf{r})}] = \begin{pmatrix} -A_{(\mathbf{c},\mathbf{r})}^\T& 0  \\ 0 & -A_{(\mathbf{c},\mathbf{r})}\end{pmatrix}.
$$
then we have
\begin{equation}\label{corchete2}
[u,v]_{(\mathbf{c},\mathbf{r})} = \phi (u,v)e_1+\phi(M_{(\mathbf{c},\mathbf{r})}u,v)e_2 ,\text{ for } u,v \in \mathbb R^{4s}.
\end{equation}
One can deduce [LT]

\begin{proposition}\label{theorem3.1}

\

(a) Every fat algebra with center of dimension 2 is isomorphic to some  $\mathfrak{n}_{(\mathbf{c},\mathbf{r})}$ with $\mathbf{c} \in \mathbb{U}^\ell$.

(b) Two of these are isomorphic if and only if the $\textbf{r}$'s coincide up to permutations and the $\textbf{c}$'s differ by  some  M\"obius transformation acting componentwise.

(c) $\mathfrak{n}_{(\mathbf{c},\mathbf{r})}$ is of type $H$ if and only if $\mathbf{c}=(c,\ldots,c)$ and $\mathbf{r}=(1,\ldots,1)$
\end{proposition}

Let now
$$\mathfrak{n} = \mathfrak{n}_{(\mathbf{c},\mathbf{r})}$$
 be fat and let $G=G(\mathfrak{n})$, etc.
We denote $\hat{\mathfrak{n}} $  the algebra obtained by replacing the matrices $A(c,r)$ by their semisimple parts and setting all $c_j=\sqrt{-1}$. The resulting $\hat A(c,r)$  consists of blocks $\begin{pmatrix} 0 & 1 \\-1 & 0  \end{pmatrix}$ along the diagonal and  $\hat{\mathfrak{n}} $  is  isomorphic to the $H$-type algebra ${\mathfrak{n}}_{((i,\ldots,i),(1,\ldots,1))}$. The correspondence
$$\mathfrak{n}\mapsto \mathfrak{\hat n}$$
is functorial and seems extendable inductively to fat algebras of any dimension, although here we will maintain the assumption $m=2$.

\medbreak

\begin{lemma}\label{theorem3.2}
 $
G_0(\mathfrak{n}) \subset G_0(\mathfrak{\hat n})
$ and $\dim\Aut( \mathfrak{n})\leq \dim\Aut( \hat{\mathfrak{n}})$.
\end{lemma}
\begin{proof}
Let $\phi$, $\psi$, $M_{(\mathbf{c},\mathbf{r})}\in\End(\mathfrak{v})$ be as above, so that
$$
\phi(M_{(\mathbf{c},\mathbf{r})}u,v) = \psi_{(\mathbf{c},\mathbf{r})}(u,v).
$$
By formula (\ref{corchete2}), $g \in G_0( \mathfrak{n}_{(\mathbf{c},\mathbf{r})})$ if and only if
$$
\phi(u,v) =\phi (gu,gv),\qquad \phi(M_{(\mathbf{c},\mathbf{r})}u,v)=\phi(M_{(\mathbf{c},\mathbf{r})}gu,gv)= \phi(g^{-1}M_{(\mathbf{c},\mathbf{r})}gu,v),
$$
i.e., if and only if  $g\in Sp(\phi)$ and commutes with $M_{(\mathbf{c},\mathbf{r})}$. In particular it commutes with the semisimple part $\hat M_{(\mathbf{c},\mathbf{r})}$.
This is conjugate  to a matrix having blocks $ Z(c)= \begin{pmatrix} \Re(c) &\Im(c) \\ -\Im(c) &\Re(c) \end{pmatrix} $ for various $c\in \mathbb{C}$ along the diagonal, and zeros elsewhere. Every matrix commuting with such a matrix will surely commute with that having all $c=1$. It follows that
$g$ also preserves  $\phi(\hat M_{(\mathbf{c},\mathbf{r})}u,v)$ and, therefore, it is an automorphism of $\hat{\mathfrak{n}}$ as well. Thus,
$$
G_0( \mathfrak{n} )\subset G_0( \hat{\mathfrak{n}}).
$$
From Corollary \ref{corollary2.3},
$
\dim(G( \mathfrak{n})/G_0( \mathfrak{n}   )) \le  \dim(G( \hat{\mathfrak{n}}  )/ G_0( \hat{\mathfrak{n}} )),
$
and therefore

$
\dim G( \mathfrak{n} ) = \dim(G( \mathfrak{n} )/G_0( \mathfrak{n} )) +\dim G_0( \mathfrak{n} )  \le  \dim(G( \hat{\mathfrak{n}})/ G_0( \hat{\mathfrak{n}}))+ \dim G_0( \hat{\mathfrak{n}}   )   =  \dim G( \hat{\mathfrak{n}}).
$

Formula (\ref{igualdad}) implies
 $\dim \Aut(\mathfrak{n})\leq \dim \Aut(\mathfrak{\hat n})$, as claimed.
 \end{proof}

\renewcommand\b{\mathbf }

\

Next we will describe  $\mathfrak g_0(\mathfrak{n}_{(c,r)})$ for $c\in \mathbb{U}$ and $r\in  \mathbb{{N}}_+$, i.e., the case when the matrices $A$ consist of a single block.
Since $c$ is $SL(2,\C)-$conjugate to $i$, it is enough to take $c=i$. Define the  $2\times 2$-matrices
$$
\b 1 = \begin{pmatrix} 1&0\\0&1\end{pmatrix},\quad \b i = \begin{pmatrix} 0&-1\\1&0\end{pmatrix}, \quad \b x=\begin{pmatrix} 0&1\\1&0\end{pmatrix},\quad
\b y = \begin{pmatrix} -1&0\\0&1\end{pmatrix},
$$
and let $M_r(\mathbb R\langle \b 1,\b i\rangle)$ and  $M_r(\mathbb R\langle \b x,\b y\rangle)$ denote the real vector spaces of $r \times r$ matrices with coefficients in the span of $\b 1,\b i$ and   $ \b x,\b y $ respectively. Then the vector space
$$
\mathcal R(r) = \left\{\begin{pmatrix} A&B\\C&D\end{pmatrix}:\; A,D \in M_r(\mathbb R\langle \b 1,\b i\rangle), \;B,C \in M_r(\mathbb R\langle \b x,\b y\rangle) \right\},
$$
is a actually a matrix algebra.

Note that
$$
\b 1^{\tt t} =\b 1,\quad \b i^{\tt t} = -\b i,\quad \b x^{\tt t} = \b x,\quad \b y^{\tt t} =\b y.
$$
Letting
 $A^{\tt t}$ denote  the transpose or an $\mathbb R$-matrix and $A^t$, $A^*$ the transpose and conjugate transpose of $\mathbb R[\b i,\b x,\b y]$-matrices, one obtains
$$
A^{\tt t} = A^*
$$ for $A \in M_r(\mathbb R\langle \b 1,\b i\rangle)$ while
$$A^{\tt t} = A^t $$
 for $A \in M_r(\mathbb R\langle \b x,\b y\rangle)$.

With the notation
$$J_1 = [\phi]\qquad J_2 = [\psi_{((i,\ldots,i),(1,\ldots,1))}],$$
$$
\mathfrak g_0(\hat{\mathfrak n}) =  \left\{ X \in M_{4r}(\mathbb R): J_1X + X^{\tt t}J_1 = 0,   J_2X + X^{\tt t}J_2 = 0\right\}.
$$
From \cite{S} we know that
$$\mathfrak g_0(\hat{\mathfrak n})\cong \mathfrak{sp}(r,\mathbb C)^\R$$
Changing basis,
$$
\mathfrak g_0(\hat{\mathfrak n}) =  \left\{X \in \mathcal R(r): \; J_1X + X^{\tt t}J_1 = 0, \;  J_2X + X^{\tt t}J_2 = 0\right\}
$$
where
$$
J_1 = \begin{pmatrix} 0&I_r\\-I_r&0\end{pmatrix}, \quad
J_2 =  \begin{pmatrix} 0&\mathbf{i}I_r\\\mathbf{i}I_r&0\end{pmatrix}.
$$
This gives an alternative description of this algebra:
$$
\mathfrak g_0(\hat{\mathfrak n}) = \left\{\begin{pmatrix} A&B\\C&-A^*\end{pmatrix}: \ A \in M_r(\mathbb R\langle \b 1,\b i\rangle),\ B,C \in M_r(\mathbb R\langle \b x,\b y\rangle),\ B^t=B,\ C^t=C\right\}
$$

We now restrict our attention to  matrices  $\begin{pmatrix} A&B\\C&-A^*\end{pmatrix}$ in  $\mathfrak g_0(\hat{\mathfrak n})$  where $A,B,C$ have the respective forms
\begin{align*}
 \begin{pmatrix}  a_{1} & a_{2} &\cdots &a_{r} \\0  & \ddots &\ddots& \vdots \\\vdots &\ddots&\ddots& a_{2}\\0  & \cdots &0  &a_{1}\end{pmatrix}
 \qquad \begin{pmatrix}  b_{1} &\cdots& b_{r-1}  &b_{r} \\  \vdots& \iddots &\iddots& 0 \\b_{r-1} &\iddots&\iddots&\vdots\\b_{r}  & 0 &\cdots  &0\end{pmatrix}
\qquad \begin{pmatrix} 0 &\cdots& 0  &c_{1 } \\\vdots  & \iddots &\iddots& c_{2} \\0 &\iddots&\iddots& \vdots\\c_{1}  & c_{2} &\cdots  &c_{r}\end{pmatrix}
\end{align*}
with coefficients in $M_2(\R)$. Let $\b A_k=\begin{pmatrix} A&0\\0&-A^*\end{pmatrix}$ having $a_{k} =\b 1$ and zero otherwise and $\b A'_k$  the matrix of the same form but with $a_{k} =\b i$ and zeros elsewhere. Similarly, let $\b B_k$ (resp. $\b C_k$) the matrix
$\begin{pmatrix}0&B\\0&0\end{pmatrix}$ (resp., $\begin{pmatrix} 0&0\\C&0\end{pmatrix}$)
with $b_{k}$ (resp. $c_{k}$) equal to $\b x$  and zeros elsewhere, and $\b B'_k$ (resp. $\b C'_k$) with $b_{k}$ (resp. $c_{k}$) equal to $\b y$ and zeros elsewhere.

\begin{theorem} \label{th3.3}Let $\mathfrak{n} = \mathfrak{n}_{({c},{r})}$, $(c,r) \in \mathbb{\mathbb{U}}\times \mathbb N$, and regard
$\mathfrak g_0(\mathfrak n)$ as a subalgebra of $\mathfrak gl(\mathfrak v)$. Then,
\begin{enumerate}
\item $\mathfrak g_0(\mathfrak n)$ is the $\mathbb R$-span of $\b A_i,\b A'_i,\b B_i,\b B'_i,\b C_i,\b C'_i$ for $1 \le i \le r$.
\item The semisimple part of ${\mathfrak g}_0(\mathfrak n) $ is the span of $\b A_1,\b A'_1,\b B_1,\b B'_1,\b C_1,\b C'_1$.
\item The solvable radical is the span of $\b A_i,\b A'_i,\b B_i,\b B'_i,\b C_i,\b C'_i$ with $1 <i\le r$.
\end{enumerate}
In particular, the $\mathbb R$-dimension de ${\mathfrak g}_0(\mathfrak{n})$ is equal to $6r$ and the semisimple part of ${\mathfrak g}_0(\mathfrak n) $ is isomorphic to $\mathfrak{sp}(1,\mathbb C)$.

\end{theorem}
\begin{proof} It is enough to consider the case $\mathfrak{n} = \mathfrak{n}_{({i},{r})}$.
Let $T_2 = [\psi_{(i,r)}]$ and write
$T_2 = J_2 + N_2$ where
$$
N_2 = \begin{pmatrix} 0&N\\-N^t&0\end{pmatrix}, \text{ with } N = \begin{pmatrix} 0 &\cdots &0 &0 \\ 1 & \ddots & 0& 0 \\ &\ddots&\ddots& \\ 0 & \cdots &1 &0\end{pmatrix}.
$$
From Lemma \ref{theorem3.2},
$
{\mathfrak g}_0(\mathfrak n) =  \left\{ X \in \mathfrak g_0(\hat{\mathfrak n}):\ T_2X + X^{\tt t}T_2 = 0\right\}
$. As $\mathfrak g_0(\mathfrak n) \subset \mathfrak g_0(\hat{\mathfrak n})$ one obtains
$$
{\mathfrak g}_0(\mathfrak n) =  \left\{ X \in \mathfrak g_0(\hat{\mathfrak n})  :N_2X + X^{\tt t}N_2 = 0\right\}.
$$
The conditions on $\begin{pmatrix} A&B\\C&-A^*\end{pmatrix} \in {\mathfrak g}_0(\mathfrak n)$ are, explicitly,
\begin{align}
0 &= NC -C^tN^t = NC - (NC)^t \label{(1)} \\
0 &= N^tA - AN^t  \label{(3)}  \\
0 &= N^tB - B^tN =  N^tB -( N^tB)^t. \label{(4)}
\end{align}

For the first equation, note that $NC$ symmetric if and only if  $c_{i,j+1} = c_{j,i+1}$ and $c_{1,j}=0$ for $i,j<n$. Since $C$ is symmetric,   $c_{i,j+1} = c_{j,i+1} =c_{i+1,j}$ and $c_{1,j}=0$ for $i,j<n$. We conclude:

If $i+j =k \le r$, $c_{i,j} = c_{i,k-i} = c_{i-1,k-i +1} = c_{i-2,k-i+2}\cdots =c_{1,k-1} =0$

If $i+j =k > r$,  $c_{i,j} = c_{i,k-i} = c_{i+1,k-i -1} =  c_{i+2,k-i-2}\cdots =c_{r,k-i+i-r} = c_{r,k-r}$

Thus, the strict upper antidiagonals are zero and each lower antidiagonal have all its elements equal.

For the second equation, note that $N^t$ and $A$ commute. This is equivalent to  $c_{i,j} = c_{t,s}$ when $j-i=s-t$ and $c_{i,1}= 0$ for $i>1$.
The first condition implies that each diagonal have all its elements equal, while the second  implies that the strict lower diagonals are zero.

Equation (\ref{(4)}) is analogous to equation (\ref{(1)}): the condition $N^tB$ symmetric is equivalent to each antidiagonal have all its elements equal and that the strict lower antidiagonals are 0.

From all this we conclude that the span of $\b A_i,\b A'_i,\b B_i,\b B'_i,\b C_i,\b C'_i$ with $1 \le i \le r$ is ${\mathfrak g}_0(\mathfrak n)$ and (1) follows.

(2) and (3) follow from (1) and the explicit presentation of the matrices  $\b A_i,\b A'_i,\b B_i,\b B'_i,\b C_i,\b C'_i$.

\end{proof}

\begin{corollary} (of the proof)
Let $\mathfrak{n}$ be fat. Then $\dim(\mathfrak g_0(\mathfrak n))$ is maximal  if and only if $\mathfrak{n}$ is of $H$-type.
\end{corollary}
\begin{proof}
Let $(\mathbf{c},\mathbf{r}) =((c_1,\ldots,c_l),(r_1,\ldots,r_l))$ be such that $\mathfrak{n}=\mathfrak{n}_{(\mathbf{c},\mathbf{r})}$.
We know that $\mathfrak g_0({\mathfrak n}) \subset \mathfrak g_0(\hat{\mathfrak n})$. If $c_i \not= c_j$ for some $i,j$, then there is not intertwining operator between
the blocks corresponding to these invariants, so $\mathfrak g_0({\mathfrak n}) \not= \mathfrak g_0(\hat{\mathfrak n})$.

When $c_1=c_2=\cdots=c_l$ we can consider $c_j =i$ for all $j$. Let $r = \sum r_i$. In this case  if $\begin{pmatrix} A&B\\C&-A^*\end{pmatrix} \in   \mathfrak g_0(\mathfrak n)$ must be satisfy the equations
(\ref{(1)}),  (\ref{(3)}), (\ref{(4)}) but with $N$ such that coefficients $n_{j+1,j}$   are $0$ or $\b 1$. Suppose now that $\mathfrak g_0(\mathfrak n)$ is not of $H$-type, then some $n_{j+1,j}$ is equal to $\b 1$. We assume that $n_{21} = \b 1$ and let  $A \in M_r(\mathbb R[\b i])$ such that $a_{12} = \b 1$ and $0$ otherwise, then
$$
X = \begin{pmatrix} A&0\\0&-A^*\end{pmatrix}
$$
belongs to  $\mathfrak g_0(\hat{\mathfrak n})$ but is not in $\mathfrak g_0({\mathfrak n})$.

\end{proof}

It can be shown in general that the semisimple part of $\mathfrak g_0(\mathfrak n)$ is isomorphic to $\oplus_i \mathfrak{sp}(m_i,\mathbb C)$, where $m_i$ is the multiplicity of the pair $(c_i,r_i)$ in $(\mathbf{c},\mathbf{r})$.

\

In the case $m=2$, $\mathfrak{g/g}_0$ is either $0$ or isomorphic to $\mathfrak{so}(2)$.

\begin{theorem}  $ \mathfrak g(\mathfrak n)/\mathfrak g_0(\mathfrak n) \cong \mathfrak{so}(2) $  if  $c_1=\cdots = c_\ell$, and $0$ otherwise.
\end{theorem}
 \begin{proof}
 $ \mathfrak {g /\mathfrak g}_0 $ is a compact subalgebra of $\mathfrak{gl}(2) $, hence of the form $g \mathfrak{so}(2) g^{-1}$ for some $g\in SL(2,\R)$ and it is nonzero if and only if there exists $X\in \mathfrak{sl}(v)$ such that, in the notation of the proof of Theorem \ref{th3.3},
 $$
 \begin{pmatrix} X&0\\0&g\mathbf{i}g^{-1}\end{pmatrix}
 $$
 is a derivation of $\mathfrak{n}$. For $g=\mathbf{1}$, if $T_1,T_2$ correspond to the standard basis of $\mathfrak{z}$, the equations  for $X$ become
 $$
\text{\em (a)}\quad T_1 X + X^\T T_1 = T_2, \qquad \text{\em (b)}\quad T_2 X + X^\T T_2 = -T_1
$$
In normal form, and for a single block $A_{(i,r)}$,
$$
T_1=J_1 = \begin{pmatrix} 0&I_r\\-I_r&0\end{pmatrix}, \quad
T_2 =  \begin{pmatrix} 0_{  }&\b iI_r +N \\\b iI_r -N^\T & 0_{  } \end{pmatrix}.
$$
We decompose
$$
T_2 = J_2+N_2,\qquad \mathrm{with}\qquad J_2 =  \begin{pmatrix} 0&\b iI_r\\\b iI_r&0\end{pmatrix},\ N_2=
\begin{pmatrix} 0_{  }& N \\ -N^\T & 0_{  } \end{pmatrix}
$$
and regard $J_1,J_2,T_1,T_2, N_2$ as matrices with coefficients in $M_2(\mathbb R)$. Note that $J_1,J_2$ correspond to  $\mathfrak{\hat n}$, of type $H$.
Let
$$
Y_0 = \begin{pmatrix}
   0&0  &0&0&0&0&s&0
\\ 0&2\b i&0&0&0&&s&0
\\ 0&\b 1&4  \b i&0&0&0& s&0
\\ 0&0&2\b 1&6\b i&0&0&s&0
\\ 0&0&0& 3 \b 1&8\b i& & s&0
\\ \vdots& \vdots&\vdots & \vdots& \vdots& \vdots& \ddots&0
\\ 0&0&0& \vdots& \vdots& 0& (n-2)\b 1&2(n-1) \b i\end{pmatrix}.
$$
A straightforward calculation shows that
$$
X_0= \begin{pmatrix} -Y_0^\T &0\\ 0& -Y^\T_0+ \b iI_r +N\end{pmatrix}
$$
is a solution of (a), (b). We conclude that
$$
 \begin{pmatrix}X_0 &0\\ 0&\b i\end{pmatrix}
$$
is a derivation of $\mathfrak{n}_{(i,r)}$, which lies in $\mathfrak{g}(\mathfrak{n}_{(i,r)})$ but not in $\mathfrak{g}_0(\mathfrak{n}_{(i,r)})$.

For any $c\in \mathbb{U}$, $\mathfrak{n}_{(c,r)}\cong \mathfrak{n}_{(i,r)}$, hence they have the same $\mathfrak{g/g}_0$ up to isomorphisms. In fact, for any $g\in Sl(2,\R)$, the algebra $\mathfrak{n}_{(g\cdot i,r)}$ has a derivation of the form
$$
 \begin{pmatrix}X &0\\ 0&g\b ig^{-1}\end{pmatrix}.
$$
For a fixed $g$, these $X$ are unique modulo $\mathfrak{g}_0$ and come in normal form. Clearly, $c$  determines the $2\times 2$ matrix $ g\b ig^{-1}$ and the complex number $g\cdot i$.

In the case of an arbitrary fat $\mathfrak{n}_{(\mathbf{c},\mathbf{r})}$, each block $(c_k,r_k)$ determines a corresponding
$X_k$ such that
$$
 \begin{pmatrix}X_k &0\\ 0& g_k\b i\g^{-1}_k\end{pmatrix}
$$
is a derivation of $\mathfrak{n}_{(c_k,r_k)}$. If $n_{(\mathbf{c},\mathbf{r})}$ has a derivation in $\mathfrak{g}$ that is not in $\mathfrak{g}_0$, then its must have one which is combination of such, acting on $\mathfrak{v}$ as $X_1+ X_2 + \cdots$. This forces
all the $ g_k\b ig_k^{-1}$ to be the same and all the $c_i$ to be the same. The reciprocal is clear.
\end{proof}

In particular, all algebras $\mathfrak{n}_{(\mathbf{c},\mathbf{r})}$ with $c_1=...=c_\ell$ and $r_i>1$ maximize de dimension of  $\mathfrak{g/g}_0$, but they are not type $H$.

Lauret had pointed out to us that there were non-type H algebras such that $ \mathfrak g(\mathfrak n)/\mathfrak g_0(\mathfrak n) \not= 0 $. Independently, Oscari proved that this holds whenever the $c_i$'s all agree.
\end{section}

\begin{section}{Fat distributions}\label{section4} Let $D$ be a smooth vector distribution on a smooth manifold $M$, i.e., a subbundle of the tangent bundle $T(M)$. Its nilpotentization, or  symbol, is the  bundle on $M$ with fiber
$$N^D(M)_p = \bigoplus_j D^{(^j)}_p / D^{(j-1)}_p $$ where
$D^{(1)}_p=D_p$ and $D^{(j+1)}_p = D^{(j)}_p + [\Gamma(D),  \Gamma(D^{j})]_p$. The Lie bracket in $\Gamma(T(M))$ induces
a graded nilpotent Lie algebra structure on each fiber of $N^D(M)$. If $D^{(j)} =T(M)$ for some $j$, $D$ is called completely non-integrable. If $D^{(2)} =T(M)$,
the nilpotentization is 2-step, which in  the notation of the previous section, is
$$\mathfrak{n}_p= N_D(M)_p = D_p \oplus  \frac{D_p+[\Gamma(D),  \Gamma(D)]_p}{D_p} = \mathfrak{v}_p+\mathfrak{z}_p,$$
It is also easy to see that $D$ is  fat in the sense of Weinstein {\cite M} if and only if $\mathfrak{n}_p=  \mathfrak{v}+\mathfrak{z}$ is non-singular, i.e., fat in the sense defined in the section 1.

A subriemannian metric $g$  defined on $D$  determines a metric on $\mathfrak{v}$. On $\mathfrak{z}$ we put a metric $\sigma$ invariant under $G$.
Let $\{\phi_1,...,\phi_m; \psi_1,...,\psi_n\}$ be a coframe on $M$ such that
$$D=\cap \ker \phi_i,$$
with $\{\phi_1,...,\phi_m\}$ and  $\{\psi_1,...,\psi_n\}$ orthonormal with respect to $g+\sigma$. Define  $T_z\in \End(D)$ as before, by
$$ \sigma(z,[u,v]) = g(T_zu,v).$$
Then $D$ is fat if and only if $T_z$ is invertible for all non-zero $z\in \mathfrak{z}$. The structure equations for the coframe can be written
$$d\phi_k \equiv \sum_i (T_k \psi_i)\wedge \psi_i \qquad mod(\phi_\ell)$$
with the $T_k$'s having the property that any non-zero linear combination of them is invertible. This is deduced from the fact that if $u,v\in \mathfrak{v}$, then $d\phi [u,v] = - \phi([u,v])$, since $u(\phi(v)) = u (0)=0$. The $d\psi$'s are essentially arbitrary.

Let now $M$ be a the simply connected Lie group with a fat Lie algebra $\mathfrak{n}$,   $D$ the left-invariant distribution on $M$ such that $D_e=\mathfrak{v}$. For a left-invariant coframe, the structure equations take the form
 $$d\phi_k = \sum_i (J_k \psi_i)\wedge \psi_i, \qquad d\psi_i =0 $$
 where $J_1,...,J_m$ are anticommuting complex structures on $D$.

 The results from the previous sections lead to consider fat distributions satisfying
$$ d\phi_k = \sum_i (J_k \psi_i)\wedge \psi_i \qquad mod(\phi_\ell) \leqno{(4.1)}$$
where the $J_k$ are sections of $\End(T(M)^*)$ satisfying the Canonical Commutation Relations
$$J_iJ_j + J_jJ_i = -2 \delta_{ij}.$$

The Equivalence Problem for these systems has been discussed for distributions with growth vector $(2n,2n+1), (4n,4n+3)$ and $(8,15)$. In these cases  $\mathfrak{n}$ is parabolic, i.e., isomorphic to the Iwasawa subalgebra of a real semisimple Lie algebra $\mathfrak{g}$ of real rank one. The Tanaka \cite{T} subriemannian prolongation of such algebra is $\mathfrak{g}$, while in the non-parabolic case
is just
$$\mathfrak{n}+ \mathfrak{k(n)}+ \mathfrak{a}(\mathfrak{n})$$
where $\mathfrak{a}(\mathfrak{n})$ the 1-dimensional Lie algebra of dilations \cite{Su}. In this case,Tanaka's theorem implies that, in the notation of \cite{Z}, the first pseudo G-structure $P^0$ already carries a canonical frame.

As this paper was being written, E. van Erp pointed out to us his article \cite{Er}, where fat distributions are called polycontact and those satisfying  (4.1) arise by imposing a compatible conformal structure.

\end{section}

\vskip .3cm
\begin{center} \bf Acknowledgments\end{center}
We wish to thank Professor J. Vargas for helpful discussions and M. Subils for pointing out a mistake in a previous version of this paper.

\end{document}